\documentclass[12pt,a4paper]{amsart}

\usepackage{amsmath,amsthm,amsfonts,amssymb,array,amscd}
\usepackage{amsmath,geometry,amssymb,amsfonts,amsthm,graphicx,enumerate,latexsym,tabularx,amscd}

\usepackage{refcheck} 
\norefnames
\nocitenames
\theoremstyle{plain}
\usepackage{amsthm}
 \newtheorem{thm}{Theorem}[section]
 \newtheorem{prop}[thm]{Proposition}
 \newtheorem{lem}[thm]{Lemma}
\theoremstyle{definition}
 \newtheorem{dfn}[thm]{Definition}
 \newtheorem{rem}[thm]{Remark}
 \numberwithin{equation}{section}
\theoremstyle{definition}
\theoremstyle{remark}
 \numberwithin{equation}{section}

\renewcommand{\leq}{\leqslant}
\renewcommand{\ge}{\geqslant}\renewcommand{\geq}{\geqslant}

\usepackage{tikz}
\usepackage[pdftex]{hyperref}
\usetikzlibrary{matrix,arrows}

\newcommand{\bbC}{\mathbb{C}}
\newcommand{\bbF}{\mathbb{F}}

\renewcommand{\and}{\quad \mbox{and} \quad}  
\renewcommand{\leq}{\leqslant}
\renewcommand{\ge}{\geqslant}\renewcommand{\geq}{\geqslant}

\title{Twisting formula of epsilon factors}



\author[Biswas]{\bfseries Sazzad Ali Biswas}

\address{
Chennai Mathematical Institute\\ 
H1, Sipcot It Park, Siruseri-603103 \\ 
India}
\email{sabiswas@cmi.ac.in, sazzad.jumath@gmail.com}

\thanks{The author is partially supported by IMU-Berlin Einstein Foundation, Berlin, Germany and CSIR, India } 

\begin{document}

\vspace{10mm}
\setcounter{page}{1}
\thispagestyle{empty}

\begin{abstract}
For characters of a non-Archimedean local field we have explicit formula for epsilon factors. But in general, 
we do not have
any generalized twisting formula of epsilon factors. In this paper we give a generalized twisting formula of 
epsilon factors via local Jacobi sums.

\end{abstract}

\maketitle

\section{\textbf{Introduction}}

By Langlands, we can associate a local epsilon factor (also known as local constant) with each character 
$\chi:F^\times\to\bbC^\times$ 
of a non-Archimedean local field $F$ of characteristic zero. 
 In general, we do not have any explicit formula of epsilon factor of a 
 twisted character (by a ramified character). 
 Let $\chi_1$ and $\chi_2$ be two characters of $F^\times$. 
 Let $\psi:F\to\bbC^\times$ be a nontrivial additive character of $F$. If any one of these characters is unramified, 
 then we have a
 formula for $\epsilon(\chi_1\chi_2,\psi)$ due to Tate. Also if conductor, $a(\chi_1)\geq 2 a(\chi_2)$, 
 then by Deligne (cf. \cite{D1}, Lemma 4.16) we have a formula for $\epsilon(\chi_1\chi_2,\psi)$. 
 In this article we give a
 generalized
 twisting formula for $\epsilon(\chi_1\chi_2,\psi)$, when both $\chi_1$ and $\chi_2$ are ramified via the following 
 local Jacobi sums.

 Let $U_F$ be the group of units in $O_F$ (ring of integers of $F$). 
For characters $\chi_1$, $\chi_2$ of $F^\times$ and a positive integer $n$, we define the 
local Jacobi sum:
\begin{equation}
  J_t(\chi_1,\chi_2,n)=\sum_{\substack{x\in \frac{U_F}{U_{F}^{n}}\\t-x\in U_F}}\chi_{1}^{-1}(x)\chi_{2}^{-1}(t-x).
 \end{equation}
 In our twisting formula,
conductors of characters play an important role and the formula (cf. Theorem \ref{Theorem 5.1}) is:
  \begin{align}
 \epsilon(\chi_1\chi_2,\psi)
 =\begin{cases}
   \frac{q^{\frac{n}{2}}\epsilon(\chi_1,\psi)\epsilon(\chi_2,\psi)}{J_1(\chi_1,\chi_2,n)} & \text{when $n=m=r$},\\ 
    \frac{q^{\frac{r}{2}}\chi_1\chi_2(\pi_{F}^{r-n})\epsilon(\chi_1,\psi)\epsilon(\chi_2,\psi)}{J_1(\chi_1,\chi_2,n)} & \text{when $n=m>r$},\\
     \frac{q^{n-\frac{m}{2}}\epsilon(\chi_1,\psi)\epsilon(\chi_2,\psi)}{J_1(\chi_1,\chi_2,n)} & \text{when $n=r>m$},\\
   \end{cases}
\end{align}
  where $n=a(\chi_1), m=a(\chi_2)$, $r=a(\chi_1\chi_2)$ and $q$ is the cardinality of the residue field of the field 
  $F$.

\section{\textbf{Notations and Preliminaries}}

Let $F$  be a non-Archimedean local field of characteristic zero
, i.e., a finite extension of the field $\mathbb{Q}_p$ (field of $p$-adic numbers),
where $p$ is a prime.
Let $O_F$ be the 
ring of integers in local field $F$ and $P_F=\pi_F O_F$ is the unique prime ideal in $O_F$ 
and $\pi_F$ is a uniformizer, i.e., an element in $P_F$ whose valuation is one i.e.,
 $\nu_F(\pi_F)=1$.
The cardinality of the residue field $\bbF_q=O_F/P_F$
of $F$ is $q$, i.e., $|\bbF_q|=q$. Let $U_F=O_F-P_F$ be the group of units in $O_F$.
Let $P_{F}^{i}=\{x\in F:\nu_F(x)\geq i\}$ and for $i\geq 0$ define $U_{F}^{i}=1+P_{F}^{i}$
(with proviso $U_{F}^{0}=U_F=O_{F}^{\times}$).

\begin{dfn}[\textbf{Conductor of characters}]
 
The conductor of any nontrivial additive character $\psi$ of a field  is an integer $n(\psi)$ if $\psi$ is trivial
on $P_{F}^{-n(\psi)}$, but nontrivial on $P_{F}^{-n(\psi)-1}$. 
We also consider $a(\chi)$ as the conductor of 
 nontrivial character $\chi: F^\times\to \mathbb{C}^\times$, i.e., $a(\chi)$ is the smallest integer $m\geq 0$ such 
 that $\chi$ is trivial
 on $U_{F}^{m}$. We say $\chi$ is {\bf unramified} if the conductor of $\chi$ is zero and otherwise {\bf ramified}.
  We also recall here that for two characters $\chi_1$ and $\chi_2$ of $F^\times$ we have 
 $a(\chi_1\chi_2)\leq\mathrm{max}(a(\chi_1),a(\chi_2))$ with equality if $a(\chi_1)\neq a(\chi_2)$. 
\end{dfn}

\subsection{Classical Gauss sums and Jacobi sums}

Let $k_q$ be a finite field of order $q$. Let $\chi, \psi$ be a multiplicative and an additive character respectively of $k_q$. 
Then the classical Gauss sum $G(\chi,\psi)$ is defined by 
\begin{equation}
 G(\chi,\psi)=\sum_{x\in k_{q}^{\times}}\chi(x)\psi(x).
\end{equation}

Let $\chi_1$ and $\chi_2$ be two multiplicative character of $k_{q}$. The classical Jacobi sum $J_1(\chi_1,\chi_2)$ is defined by 
\begin{equation}
 J_1(\chi_1,\chi_2)=\sum_{x\in k_{q}^{\times}}\chi_1(x)\chi_2(1-x).
\end{equation}

The relation between classical Gauss sums and Jacobi sums is as follows (cf. \cite{BRK}, p. 59, Theorem 2.1.3)\\
\begin{enumerate}
 \item Let $\chi_1, \chi_2$ be two multiplicative character of $k_q$ and $\psi$  a nontrivial
 additive character of $k_q$. If $\chi_1\chi_2$ is nontrivial, then 
 \begin{equation}
  J_1(\chi_1,\chi_2)=\frac{G(\chi_1,\psi)\cdot G(\chi_2,\psi)}{G(\chi_1\chi_2,\psi)}.
 \end{equation}
\item If $\chi_1,\chi_2$ and $\chi_1\chi_2$ are all nontrivial, then we have 
\begin{equation}\label{eqn 2.4}
 |J_1(\chi_1,\chi_2)|=q^{\frac{1}{2}}.
\end{equation}

\end{enumerate}

\subsection{\textbf{Epsilon factors}}

For a nontrivial multiplicative character $\chi$ of $F^\times$ and nontrivial additive character $\psi$ of $F$,
we have (cf. \cite{RL}, p. 4)
\begin{equation}\label{label1}
 \epsilon(\chi,\psi,c)=\chi(c)\frac{\int_{U_F}\chi^{-1}(x)\psi(x/c) dx}{|\int_{U_F}\chi^{-1}(x)\psi(x/c) dx|}
\end{equation}
where the Haar measure $dx$ on $F$ is normalized such that the Haar measure $m'$ of $O_F$ is $1$, i.e., $m'(O_F)=1$, and 
 $c\in F^\times$ with valuation $n(\psi)+a(\chi)$.
To get modified \textbf{summation formula}
of epsilon factor from the integral formula (\ref{label1}), we need the next lemma which can be found in \cite{GGS}.

\begin{lem}\label{Lemma 3.1}
Let $F$ be a non-Archimedean local field with $q$ as the cardinality of the residue field of $F$.
Let $\chi$ be a nontrivial character of $F^\times$ with conductor $a(\chi)$. Let $\psi$ be an additive character of 
$F$ with conductor $n(\psi)$. We define the integration for an integer $m\in \mathbb{Z}$:
\begin{equation}
 I(m)=\int_{U_F}\chi^{-1}(x)\psi(\frac{x}{\pi_{F}^{l+m}})dx,\quad\text{where $l=a(\chi)+n(\psi)$}.
\end{equation}
Then
\begin{align}
 |I(m)|=\begin{cases}
         q^{-\frac{a(\chi)}{2}} & \text{when $m=0$,}\\
         0 & \text{otherwise.}
        \end{cases}
\end{align}
\end{lem}

Again since $m'(O_F)=1$, we have $m'(U_F^n)=q^{-n}$ for any positive integer $n$. By using the above 
Lemma \ref{Lemma 3.1}, 
 the formula (\ref{label1}) can be reduced to 
 \begin{align}
  \epsilon(\chi,\psi,c)\nonumber
  &=\chi(c)q^{a(\chi)/2}\sum_{x\in\frac{U_F}{U_{F}^{a(\chi)}}}\chi^{-1}(x)\psi(x/c)m'(U_{F}^{a(\chi)})\\
  &=\chi(c)q^{-a(\chi)/2}\sum_{x\in\frac{U_F}{U_{F}^{a(\chi)}}}\chi^{-1}(x)\psi(x/c),\label{eqn 3.11}
 \end{align}
where $c=\pi_{F}^{a(\chi)+n(\psi)}$.

Now if $u\in U_F$ is unit and replace $c=cu$, then we have 
\begin{equation}
 \epsilon(\chi,\psi,cu)=\chi(c)q^{-\frac{a(\chi)}{2}}\sum_{x\in\frac{U_F}{U_{F}^{a(\chi)}}}\chi^{-1}(x/u)\psi(x/cu)=\epsilon(\chi,\psi,c).
\end{equation}
Therefore $\epsilon(\chi,\psi,c)$ \textbf{depends} only on the exponent $\nu_{F}(c)=a(\chi)+n(\psi)$. Therefore we can 
simply write $\epsilon(\chi,\psi, c)=\epsilon(\chi,\psi)$, because $c$ is determined by 
$\nu_F(c)=a(\chi)+n(\psi)$ up to a unit $u$ which has \textbf{no influence on} $\epsilon(\chi,\psi,c)$.
If $\chi$ is unramified, i.e., $a(\chi)=0$, therefore $\nu_F(c)=n(\psi)$. Then from the formula of $\epsilon(\chi,\psi,c)$, we can write
\begin{equation}\label{eqn 3.15}
 \epsilon(\chi,\psi,c)=\chi(c),
\end{equation}
and therefore $\epsilon(1,\psi,c)=1$ if $\chi=1$ is the trivial character.

\subsection{\textbf{Known twisting formula of abelian epsilon factors:}}

\begin{enumerate}
 \item If $\chi_1$ and $\chi_2$ are two unramified characters of $F^\times$ and $\psi$ be 
 a nontrivial additive character of $F$, then we have from equation (\ref{eqn 3.15})
 \begin{equation}
  \epsilon(\chi_1\chi_2,\psi)=\epsilon(\chi_1,\psi)\epsilon(\chi_2,\psi).
 \end{equation}
 \item  Let $\chi_1$ and $\chi_2$ be ramified and unramified character of $F^\times$ respectively, 
 then (cf. \cite{JT2}, (3.2.6.3))
\begin{equation}
 \epsilon(\chi_1\chi_2,\psi)=\chi_2(\pi_F)^{a(\chi_1)+n(\psi)}\cdot \epsilon(\chi_1,\psi).
\end{equation}

\item 
 We also have a twisting formula of epsilon factor by Deligne (cf. \cite{D1}, Lemma 4.16)
 under some special condition and which is in the following:\\
Let $\alpha$ and $\beta$ be two multiplicative characters of a local field $F$ such that $a(\alpha)\geq 2\cdot a(\beta)$.
Let $\psi$ be an additive character of $F$.
Let $y_{\alpha,\psi}$ be an element of $F^\times$ such that 
$$\alpha(1+x)=\psi(y_{\alpha,\psi}x)$$
for all $x\in F$ with valuation $\nu_F(x)\geq\frac{a(\alpha)}{2}$ (if $a(\alpha)=0$, $y_{\alpha,\psi}=\pi_{F}^{-n(\psi)}$). Then 
\begin{equation}\label{eqn 3.26}
 \epsilon(\alpha\beta,\psi)=\beta^{-1}(y_{\alpha,\psi})\cdot \epsilon(\alpha,\psi).
\end{equation}

\end{enumerate}

\section{ \textbf{Generalized twisting formula of epsilon factors}}

\subsection{\textbf{Local Gauss sum}}

Let $m$ be a nonzero positive integer.
Let $\chi$ be a nontrivial multiplicative character of $F$ with conductor $a(\chi)$ and $\psi:F\to\mathbb{C}^\times$ 
be an additive character of $F$ with conductor $n(\psi)$. We define the local character sum of a character
$\chi$:
\begin{equation}\label{eqn 4.1}
 G(\chi,\psi,m)=\sum_{x\in\frac{U_F}{U_{F}^{m}}}\chi^{-1}(x)\psi(x/c),
\end{equation}
where $c=\pi_{F}^{a(\chi)+n(\psi)}$. When $m=a(\chi)$, we call $G(\chi,\psi,a(\chi))$  
as the \textbf{local Gauss sum} of character $\chi$.

Now let conductor of $\chi$ be $1$, i.e., $\chi:F/U_{F}^{1}\to\bbC^\times$. Hence $\chi|_{U_F}$ is a 
character of $U_F/U_{F}^{1}$. If the conductor of $n(\psi)$ is $-1$, i.e., $\psi:F/P_F\to\bbC^\times$,
then $\psi|_{O_F}$ is an additive character of $O_F/P_F$. Moreover, when $a(\chi)=1$ and $n(\psi)=-1$, we have 
$c=\pi_{F}^{a(\chi)+n(\psi)}=\pi_{F}^{1-1}=1$, and hence equation (\ref{eqn 4.1}) reduces to the classical
Gauss sum, i.e., $G(\chi,\psi,1)=G(\tilde{\chi},\psi')$, where $\tilde{\chi}:=\chi^{-1}|_{U_F}$ and $\psi':=\psi|_{O_F}$.

\begin{prop}\label{Proposition 4.1}
 The definition of local Gauss sum $G(\chi,\psi, a(\chi))$ does not depend on the choice of the coset 
representatives of $U_F$ mod $U_{F}^{a(\chi)}$.
\end{prop}
\begin{proof}
It is very easy to see from the definition of local Gauss sum.
If we change one of the coset representatives $x$ to $xu$ where $u\in U_{F}^{a(\chi)}$ in $G(\chi,\psi, a(\chi))$ and we have 
\begin{align*}
 G(\chi,\psi,a(\chi))
 &=\sum_{x\in\frac{U_F}{U_{F}^{a(\chi)}}}\chi^{-1}(x)\psi(x/c)\\
 &=\sum_{x\in\frac{U_F}{U_{F}^{a(\chi)}}}\chi^{-1}(xu)\psi(xu/c)\quad\text{replacing $x$ by $xu$, $u\in U_{F}^{a(\chi)}$}\\
 &=\sum_{x\in\frac{U_F}{U_{F}^{a(\chi)}}}\chi^{-1}(x)\psi(x/c)\psi(\frac{x}{c}(u-1))\\
 &=\sum_{x\in\frac{U_F}{U_{F}^{a(\chi)}}}\chi^{-1}(x)\psi(x/c).
\end{align*}
 Since $P_{F}^{-n(\psi)}$ is a fractional ideal of $O_F$, then $\frac{x}{c}(u-1)\in P_{F}^{-n(\psi)}$ for $x\in U_F$.
 Therefore $\psi(\frac{x}{c}(u-1))=1$ for all $x\in U_F$ and $u\in U_{F}^{a(\chi)}$. This proves the local 
 Gauss sum is independent 
 of the choice of the coset representatives of $U_F$ mod $U_{F}^{a(\chi)}$.
 Therefore definition of local Gauss sum does not depend on the choice of the coset representatives of $x$. 
\end{proof}
 
In the next proposition we compute the absolute value of $G(\chi,\psi,a(\chi))$ by using Lemma \ref{Lemma 3.1}.
\begin{prop}\label{Proposition 4.2}
 If $\chi$ is a ramified character of $F^\times$, then 
\begin{equation*}
 |G(\chi,\psi,a(\chi))|=q^{\frac{a(\chi)}{2}}.
\end{equation*}
\end{prop}
\begin{proof}
We can write 
\begin{equation}
 \int_{U_F}\chi^{-1}(x)\psi(x/c)dx=\sum_{x\in\frac{U_F}{U_{F}^{a(\chi)}}}\chi^{-1}(x)\psi(x/c)\times m'(U_{F}^{a(\chi)}).
 \label{label5}
\end{equation}
where $c=\pi_{F}^{a(\chi)+n(\psi)}$ and $m'$ is the Haar measure which is normalized so that 
$m'(O_F)=1$. Now from equation (\ref{label5}) we have 
\begin{align*}
 | \int_{U_F}\chi^{-1}(x)\psi(x/c)dx|
 &=|\sum_{x\in\frac{U_F}{U_{F}^{a(\chi)}}}\chi^{-1}(x)\psi(x/c)|\times|m'(U_{F}^{a(\chi)})|\\
 &=|G(\chi,\psi,a(\chi))|q^{-a(\chi)},
\end{align*}
since $m'(U_{F}^{a(\chi)})=q^{-a(\chi)}$.
Therefore from Lemma \ref{Lemma 3.1} we have 
\begin{equation}
 |G(\chi,\psi,a(\chi))|=q^{\frac{a(\chi)}{2}}.
\end{equation}
\end{proof}

Furthermore, from Lemma \ref{Lemma 3.1} it can be proved that 
\begin{equation}
 \sum_{x\in\frac{U_F}{U_{F}^{a(\chi)}}}\chi^{-1}(x)\psi(\frac{x}{\pi_{F}^{l+m}})=0,
\end{equation}
for nonzero integers $m\neq 0$ and $l=a(\chi)+n(\psi)$.

In the next lemma, we see the relation between two local character sums $G(\chi,\psi,n_1 )$, $G(\chi,\psi, n_2)$, 
where $n_1>n_2$ and here we mean 
\begin{equation*}
 G(\chi,\psi,n_i)=\sum_{x\in\frac{U_F}{U_{F}^{n_i}}}\chi^{-1}(x)\psi(x/c),\quad\text{$i=1,2$}
\end{equation*}
 where $c=\pi_{F}^{a(\chi)+n(\psi)}$. 
\begin{lem}\label{Lemma 4.3}
\begin{equation}
 G(\chi,\psi,n_1)=q^{m}G(\chi,\psi,n_2).
\end{equation}
where $m=n_1-n_2$.
\end{lem}
\begin{proof}
  From Lemma \ref{Lemma 3.1}, it is straight forward. We have 
  \begin{align*}
   \int_{U_F}\chi^{-1}(x)\psi(\frac{x}{c})dx
   &=\sum_{x\in\frac{U_F}{U_{F}^{n_1}}}\chi^{-1}(x)\psi(\frac{x}{c})\times m'(U_{F}^{n_1})\\
   &=q^{-n_1}\sum_{x\in\frac{U_F}{U_{F}^{n_1}}}\chi^{-1}(x)\psi(\frac{x}{c}),\quad\text{since $m'(U_{F}^{n_1})=q^{-n_1}$}\\
   &=q^{-n_1}G(\chi,\psi,n_1).
  \end{align*}
Similarly, we can express
\begin{equation*}
 \int_{U_F}\chi^{-1}(x)\psi(\frac{x}{c})dx=q^{-n_2}G(\chi,\psi,n_2).
\end{equation*}

 Comparing these two above equations we obtain
 \begin{equation*}
  G(\chi,\psi,n_1)=q^{n_1-n_2}G(\chi,\psi,n_2)=q^{m}G(\chi,\psi,n_2).
 \end{equation*}

\end{proof}

\subsection{\textbf{Local Jacobi sum}}

Let $\chi_1$ and $\chi_2$ be two nontrivial characters of $F^\times$. 
For any $t\in \frac{U_F}{U_{F}^{n}}$, where 
 $n\geq1$, we define the local Jacobi sum for characters $\chi_1$ and $\chi_2$:
 \begin{equation}
  J_t(\chi_1,\chi_2,n)=\sum_{\substack{x\in \frac{U_F}{U_{F}^{n}}\\t-x\in U_F}}\chi_{1}^{-1}(x)\chi_{2}^{-1}(t-x).
 \end{equation}
 When $n=1$, $t=1$, and conductors $a(\chi_1)=a(\chi_2)=1$, this local Jacobi sum is nothing but the classical Jacobi sum for 
 the characters $\chi_{1}^{-1}$ and $\chi_{2}^{-1}$, i.e., 
 $J_1(\chi_1,\chi_2,1)=J_1(\chi_{1}^{-1},\chi_{2}^{-1})$.

\begin{prop}\label{Proposition 4.4} 
\begin{equation}
 J_1(\chi_1,\chi_2,n)=\chi_1\chi_2(t)\cdot J_t(\chi_1,\chi_2,n), \quad\text{for any $t\in \frac{U_F}{U_{F}^{n}}$}.
\end{equation}
\end{prop}
\begin{proof}
 For any $t\in \frac{U_F}{U_{F}^{n}}$, from the definition of Jacobi sum, we have
 \begin{align*}
J_t(\chi_1,\chi_2,n)
&=\sum_{\substack{x\in \frac{U_F}{U_{F}^{n}}\\t-x\in U_F}}\chi_{1}^{-1}(x)\chi_{2}^{-1}(t-x)\\
&=\sum_{\substack{x/t\in \frac{U_F}{U_{F}^{n}}\\1-x/t\in U_F}}(\chi_1\chi_2)^{-1}(t)\chi_{1}^{-1}(x/t)\chi_{2}^{-1}(1-x/t)\\
&=(\chi_1\chi_2)^{-1}(t)\sum_{\substack{s=x/t\in \frac{U_F}{U_{F}^{n}}\\1-s\in U_F}}\chi_{1}^{-1}(s)\chi_{2}^{-1}(1-s)\\
&=(\chi_1\chi_2)^{-1}(t)J_1(\chi_1,\chi_2,n).
\end{align*}
Therefore 
\begin{equation*}
J_1(\chi_1,\chi_2,n)=\chi_1\chi_2(t)\cdot J_t(\chi_1,\chi_2,n), \quad\text{for any $t\in \frac{U_F}{U_{F}^{n}}$}.
\end{equation*}

\end{proof}

In the following theorem, we give a generalized twisting formula of epsilon factors via the above local Jacobi sums.

\begin{thm}\label{Theorem 5.1} 
Let $F$ be a non-Archimedean local field with $q$ as the cardinality of the residue field of $F$.
Let $\psi$ be a nontrivial additive character of $F$.
Let $\chi_1$ and $\chi_2$ be two ramified 
characters of $F^\times$ with conductors $n$ and $m$ respectively. Let $r$ be the conductor of character $\chi_1\chi_2$. Then
\begin{align}\label{eqn 4.8}
 \epsilon(\chi_1\chi_2,\psi)
 =\begin{cases}
   \frac{q^{\frac{n}{2}}\epsilon(\chi_1,\psi)\epsilon(\chi_2,\psi)}{J_1(\chi_1,\chi_2,n)} & \text{when $n=m=r$},\\ 
    \frac{q^{\frac{r}{2}}\chi_1\chi_2(\pi_{F}^{r-n})\epsilon(\chi_1,\psi)\epsilon(\chi_2,\psi)}{J_1(\chi_1,\chi_2,n)} & \text{when $n=m>r$},\\
     \frac{q^{n-\frac{m}{2}}\epsilon(\chi_1,\psi)\epsilon(\chi_2,\psi)}{J_1(\chi_1,\chi_2,n)} & \text{when $n=r>m$},\\
   \end{cases}
\end{align}
\end{thm}
\begin{proof}
We know that the formula of epsilon of a character $\chi$ of $F^\times$ is:
\begin{align}\label{eqn 4.9}
 \epsilon(\chi,\psi)=\chi(c)q^{-\frac{a(\chi)}{2}}\sum_{x\in\frac{U_F}{U_{F}^{a(\chi)}}}\chi^{-1}(x)\psi(x/c)=
 \chi(c)q^{-\frac{a(\chi)}{2}}G(\chi,\psi,a(\chi)),
\end{align}
where $c=\pi_{F}^{a(\chi)+n(\psi)}$.

 Now we divide this proof into three cases.\\
 \textbf{Case-1: When $n=m=r$.} By the definition of epsilon factor, in this case we have 
 \begin{equation}
 \epsilon(\chi_1,\psi)=\chi_{1}(c_1)q^{-\frac{n}{2}}\sum_{x\in\frac{U_F}{U_{F}^{n}}}\chi_{1}^{-1}(x)\psi(x/c_1),\label{label7}
\end{equation}
 and 
  \begin{equation}
 \epsilon(\chi_2,\psi)=\chi_{2}(c_2)q^{-\frac{m}{2}}\sum_{y\in\frac{U_F}{U_{F}^{m}}}\chi_{2}^{-1}(y)\psi(y/c_2),\label{label8}
\end{equation}
 where $c_1=\pi_{F}^{n+n(\psi)}$ and $c_2=\pi_{F}^{m+n(\psi)}$. Since $n=m$, then we consider
 $c=c_1=c_2$.
 
Now from equations (\ref{label7}) and (\ref{label8}) we have
 \begin{align}
\epsilon(\chi_1,\psi)\epsilon(\chi_2,\psi)\nonumber
&=q^{-n}\chi_{1}\chi_{2}(c)\sum_{x\in\frac{U_F}{U_{F}^{n}}}\chi_{1}^{-1}(x)\psi(x/c)\times 
\sum_{y\in\frac{U_F}{U_{F}^{m}}}\chi_{2}^{-1}(y)\psi(y/c)\\\nonumber
&=q^{-n}\chi_{1}\chi_{2}(c)\sum_{x, y\in\frac{U_F}{U_{F}^{n}}}\chi_{1}^{-1}(x)\chi_{2}^{-1}(y)\psi(x/c)\psi(y/c)\\\nonumber
&=q^{-n}\chi_{1}\chi_{2}(c)\sum_{x, y\in\frac{U_F}{U_{F}^{n}}}\chi_{1}^{-1}(x)\chi_{2}^{-1}(y)\psi(\frac{x+y}{c})\\\nonumber
&=q^{-n}\chi_{1}\chi_{2}(c)\sum_{\substack{x\in\frac{U_F}{U_{F}^{n}}\\t-x\in \frac{U_F}{U_{F}^{n}}}}
\chi_{1}^{-1}(x)\chi_{2}^{-1}(t-x)\psi(\frac{t}{c})\\\nonumber
&=q^{-n}\chi_{1}\chi_{2}(c)\sum_{a=0}^{n-1}\{\sum_{u\in \frac{U_F}{U_{F}^{n-a}}}\{\sum_{x\in\frac{U_F}{U_{F}^{n}}}\chi_{1}^{-1}(\frac{x}{\pi_{F}^{a}u})
\chi_{2}^{-1}(1-\frac{x}{\pi_{F}^{a}u})\}
(\chi_1\chi_2)^{-1}(\pi_{F}^{a}u)\psi(\frac{\pi_{F}^{a}u}{c})\}\\\nonumber
&=q^{-n}\chi_{1}\chi_{2}(c)\sum_{a=0}^{n-1}\{\sum_{u\in \frac{U_F}{U_{F}^{n-a}}}\{\sum_{s=x/u\in\frac{U_F}{U_{F}^{n-a}}}\chi_{1}^{-1}(\frac{s}{\pi_{F}^{a}})
\chi_{2}^{-1}(1-\frac{s}{\pi_{F}^{a}})\}
(\chi_1\chi_2)^{-1}(\pi_{F}^{a}u)\psi(\frac{\pi_{F}^{a}u}{c})\}\\\nonumber
&=q^{-n}\chi_{1}\chi_{2}(c)\sum_{a=0}^{n-1}\{\sum_{s\in\frac{U_F}{U_{F}^{n-a}}}\chi_{1}^{-1}(\frac{s}{\pi_{F}^{a}})
\chi_{2}^{-1}(1-\frac{s}{\pi_{F}^{a}})\times\sum_{u\in \frac{U_F}{U_{F}^{n-a}}}
(\chi_1\chi_2)^{-1}(\pi_{F}^{a}u)\psi(\frac{\pi_{F}^{a}u}{c})\}\\
&=q^{-n}\chi_{1}\chi_{2}(c)\sum_{a=0}^{n-1}\{J_{1}^{'}(\chi_1,\chi_2,a)\times G'(\chi_1\chi_2,\psi,a)\},\label{label9}
\end{align}
where 
\begin{equation}
 G'(\chi_1\chi_2,\psi,a)=\sum_{u\in \frac{U_F}{U_{F}^{n-a}}}
(\chi_1\chi_2)^{-1}(\pi_{F}^{a}u)\psi(\frac{\pi_{F}^{a}u}{c}),\label{label10}
\end{equation}
and 
\begin{equation}
 J_{1}^{'}(\chi_1,\chi_2,a)=\sum_{s\in\frac{U_F}{U_{F}^{n-a}}}\chi_{1}^{-1}(\frac{s}{\pi_{F}^{a}})
\chi_{2}^{-1}(1-\frac{s}{\pi_{F}^{a}})\label{eqn 4.14}
\end{equation}
In the above calculations, we assume $t=x+y$, where 
both  $x$ and $y$ are in $\frac{U_F}{U_{F}^{n}}$ and this $t$ can be written as 
$t=\pi_{F}^{a}u$, where $a$ varies over $\{0,1,\cdots,n-1\}$ and $u\in \frac{U_F}{U_{F}^{n-a}}$.
Furthermore, by using Lemma \ref{Lemma 3.1}, for $a\neq0$, we have
\begin{equation*}
\sum_{u\in \frac{U_F}{U_{F}^{n-a}}}(\chi_1\chi_2)^{-1}(u)\psi(\frac{\pi_{F}^{a}u}{c})=0.
\end{equation*}
Therefore, for $a\neq0$, we can write
\begin{align*}
G'(\chi_1\chi_2,\psi,a)
&=\sum_{u\in \frac{U_F}{U_{F}^{n-a}}}(\chi_1\chi_2)^{-1}(\pi_{F}^{a}u)\psi(\frac{\pi_{F}^{a}u}{c})\\
&=\chi_1\chi_2(\pi_{F}^{-a})\sum_{u\in \frac{U_F}{U_{F}^{n-a}}}(\chi_1\chi_2)^{-1}(u)\psi(\frac{\pi_{F}^{a}u}{c})\\
&=0.
\end{align*}

Therefore, we have to take $a=0$, because left side of equation (\ref{label9}) is \textbf{nonzero},
therefore $t=u\in \frac{U_F}{U_{F}^{n}}$.
Now, put $a=0$ in equation (\ref{label10}), then  we have from equation (\ref{label9})
\begin{align*}
\epsilon(\chi_1,\psi)\epsilon(\chi_2,\psi)
&=q^{-n}\chi_{1}\chi_{2}(c)J_{1}^{'}(\chi_1\chi_2,a) \sum_{\alpha\in\frac{U_F}{U_{F}^{r}}}(\chi_1\chi_2)^{-1}(\alpha)
\psi(\frac{\alpha}{c})\\
&=q^{-n}\chi_{1}\chi_{2}(c)J_1(\chi_1,\chi_2,n)G(\chi_1\chi_2,\psi,n)\\
&=q^{-\frac{n}{2}}J_1(\chi_1,\chi_2,n)\epsilon(\chi_1\chi_2,\psi).
\end{align*}
Therefore, in this case we have 
\begin{equation}
 \epsilon(\chi_1\chi_2,\psi)=\frac{q^{\frac{n}{2}}\epsilon(\chi_1,\psi)\epsilon(\chi_2,\psi)}{J_1(\chi_1,\chi_2,n)}. 
\end{equation}
\textbf{Case-2: When $n=m>r$.} Like case-1, in this case, it can be showed that $t=x+y\in \frac{U_F}{U_{F}^{n}}$ when 
$x,y\in \frac{U_F}{U_{F}^{n}}$. Since $c_1=c_2$, let $c=c_1=c_2$. In this situation we have:
\begin{align*}
 \epsilon(\chi_1,\psi)\epsilon(\chi_2,\psi)
 &=q^{-n}\chi_{1}\chi_{2}(c)\sum_{x\in\frac{U_F}{U_{F}^{n}}}\chi_{1}^{-1}(x)\psi(x/c)\times 
\sum_{y\in\frac{U_F}{U_{F}^{m}}}\chi_{2}^{-1}(y)\psi(y/c)\\
&=q^{-n}\chi_{1}\chi_{2}(c)\sum_{x, y\in\frac{U_F}{U_{F}^{n}}}\chi_{1}^{-1}(x)\chi_{2}^{-1}(y)\psi(x/c)\psi(y/c)\\
&=q^{-n}\chi_{1}\chi_{2}(c)\sum_{x, y\in\frac{U_F}{U_{F}^{n}}}\chi_{1}^{-1}(x)\chi_{2}^{-1}(y)\psi(\frac{x+y}{c})\\
&=q^{-n}\chi_{1}\chi_{2}(c)\sum_{t,x\in\frac{U_F}{U_{F}^{n}}}\chi_{1}^{-1}(x)\chi_{2}^{-1}(t-x)\psi(\frac{t}{c})\\
&=q^{-n}\chi_{1}\chi_{2}(c)\sum_{t\in \frac{U_F}{U_{F}^{n}}}\{\sum_{x\in\frac{U_F}{U_{F}^{n}}}\chi_{1}^{-1}(x/t)\chi_{2}^{-1}(1-x/t)\}
(\chi_1\chi_2)^{-1}(t)\psi(\frac{t}{c})\\
&=q^{-n}\chi_{1}\chi_{2}(c)\sum_{s=x/t \in\frac{U_F}{U_{F}^{n}}}\chi_{1}^{-1}(s)\chi_{2}^{-1}(1-s)\times\sum_{t\in \frac{U_F}{U_{F}^{n}}}
(\chi_1\chi_2)^{-1}(t)\psi(\frac{t}{c})\\
&=q^{-n}\chi_{1}\chi_{2}(c)J_1(\chi_1,\chi_2,n)\times G(\chi_1\chi_2,\psi,n)\\
&=q^{-n}\chi_{1}\chi_{2}(c)J_1(\chi_1,\chi_2,n)\times q^{n-r}
G(\chi_1\chi_2,\psi,r)\quad\text{using Lemma 4.3}\\
&=q^{-\frac{r}{2}}\chi_{1}\chi_{2}(\pi_{F}^{n-r})J_1(\chi_1,\chi_2,n)
\chi_1\chi_2(\pi_{F}^{r+n(\psi)})q^{-\frac{r}{2}}G(\chi_1\chi_2,\psi,r)\\
&=q^{-\frac{r}{2}}\chi_{1}\chi_{2}(\pi_{F}^{n-r})J_1(\chi_1,\chi_2,n)\epsilon(\chi_1\chi_2,\psi), \quad\text{since $a(\chi_1\chi_2)=r$}.
\end{align*}
Therefore, in this condition we have:
\begin{equation}
 \epsilon(\chi_1\chi_2,\psi)
 =\frac{q^{\frac{r}{2}}\chi_1\chi_2(\pi_{F}^{r-n})\epsilon(\chi_1,\psi)\epsilon(\chi_2,\psi)}{J_1(\chi_1,\chi_2,n)}.
\end{equation}
\textbf{Case-3: When $n=r>m$.} If conductor $a(\chi_1)>a(\chi_2)$, then conductor 
$a(\chi_1\chi_2)=\mathrm{max}(a(\chi_1),a(\chi_2))=a(\chi_1)$. 
Therefore, we are in this situation: $n=r>m$. In this case $c_1$ can be written as $c_1=c_2\pi_{F}^{n-m}$.
If $x,z\in \frac{U_F}{U_{F}^{n}}$, then $x+\pi_{F}^{n-m}z\in \frac{U_F}{U_{F}^{n}}$.
\begin{align*}
 \epsilon(\chi_1,\psi)\epsilon(\chi_2,\psi)
 &=q^{-\frac{m+n}{2}}\chi_1(c_1)\chi_2(c_2)\sum_{x\in\frac{U_F}{U_{F}^{n}}}\chi_{1}^{-1}(x)\psi(x/c_1)\times 
\sum_{z\in\frac{U_F}{U_{F}^{m}}}\chi_{2}^{-1}(z)\psi(z/c_2)\\
&=q^{-\frac{m+n}{2}}\chi_2(\pi_{F}^{m-n})\chi_1\chi_2(c_1)\sum_{x\in\frac{U_F}{U_{F}^{n}}}\chi_{1}^{-1}(x)\psi(x/c_1)\times 
\sum_{z\in\frac{U_F}{U_{F}^{m}}}\chi_{2}^{-1}(z)\psi(\frac{z\pi_{F}^{n-m}}{c_1})\\
&=q^{-\frac{m+n}{2}}\chi_2(\pi_{F}^{m-n})\chi_1\chi_2(c_1)\sum_{x\in\frac{U_F}{U_{F}^{n}}}\chi_{1}^{-1}(x)\psi(x/c_1)\times q^{m-n}
\sum_{z\in\frac{U_F}{U_{F}^{n}}}\chi_{2}^{-1}(z)\psi(\frac{z\pi_{F}^{n-m}}{c_1})\\
&=q^{\frac{m}{2}-\frac{3n}{2}}\chi_2(\pi_{F}^{m-n})\chi_1\chi_2(c_1)\sum_{x, z\in\frac{U_F}{U_{F}^{n}}}\chi_{1}^{-1}(x)\chi_{2}^{-1}(z)\psi(x/c_1)\psi(\frac{z\pi_{F}^{n-m}}{c_1})\\
&=q^{\frac{m}{2}-\frac{3n}{2}}\chi_2(\pi_{F}^{m-n})\chi_1\chi_2(c_1)\sum_{x, z\in\frac{U_F}{U_{F}^{n}}}\chi_{1}^{-1}(x)\chi_{2}^{-1}(z)\psi(\frac{x+z\pi_{F}^{n-m}}{c_1})\\
&=q^{\frac{m}{2}-\frac{3n}{2}}\chi_2(\pi_{F}^{m-n})\chi_1\chi_2(c_1)\chi_2(\pi_{F}^{n-m})\sum_{x, t\in\frac{U_F}{U_{F}^{n}}}\chi_{1}^{-1}(x)\chi_{2}^{-1}(t-x)
\psi(\frac{t}{c_1})\\
&=q^{\frac{m}{2}-\frac{3n}{2}}\chi_1\chi_2(c_1)\sum_{t\in \frac{U_F}{U_{F}^{n}}}\{\sum_{x\in\frac{U_F}{U_{F}^{n}}}\chi_{1}^{-1}(x/t)\chi_{2}^{-1}(1-x/t)\}
(\chi_1\chi_2)^{-1}(t)\psi(\frac{t}{c_1})\\
&=q^{\frac{m}{2}-\frac{3n}{2}}\chi_1\chi_2(c_1)\sum_{s=x/t \in\frac{U_F}{U_{F}^{n}}}\chi_{1}^{-1}(s)\chi_{2}^{-1}(1-s)\times\sum_{t\in \frac{U_F}{U_{F}^{n}}}
(\chi_1\chi_2)^{-1}(t)\psi(\frac{t}{c_1})\\
&=q^{\frac{m}{2}-\frac{3n}{2}}\chi_1\chi_2(c_1)J_1(\chi_1,\chi_2,n)\times \sum_{t\in \frac{U_F}{U_{F}^{n}}}(\chi_1\chi_2)^{-1}(t)\psi(\frac{t}{c_1})\\
&=q^{\frac{m}{2}-n}J_1(\chi_1,\chi_2,n)\chi_1\chi_2(c_1)q^{-\frac{n}{2}}\times G(\chi_1\chi_2,\psi,n)\\
&=q^{\frac{m}{2}-n}J_1(\chi_1,\chi_2,n)\epsilon(\chi_1\chi_2,\psi).
\end{align*}
Therefore we have the formula:
\begin{equation}
\epsilon(\chi_1\chi_2,\psi)=\frac{q^{n-\frac{m}{2}}\epsilon(\chi_1,\psi)\epsilon(\chi_2,\psi)}{J_1(\chi_1,\chi_2,n)}.
\end{equation}
  
\end{proof}

\begin{rem}
 Let $\chi_1, \chi_2$ be two characters of $F^\times$ with conductors $a(\chi_1)=a(\chi_2)=1$. Let $\psi$ be a nontrivial
 additive character of $F$. If the conductor of $\chi_1\chi_2$ is $1$, then from using the above Theorem \ref{Theorem 5.1}
 and equation (\ref{eqn 2.4}) we can say:
 \begin{equation}
  \frac{\epsilon(\chi_1\chi_2,\psi)}{\epsilon(\chi_1,\psi)\epsilon(\chi_2,\psi)}=\gamma,
 \end{equation}
where $\gamma$ is a root of unity for which $J_1(\chi_1,\chi_2,1)=q^{\frac{1}{2}}\cdot\gamma^{-1}$.

Now let $\mu$ be the group of roots of unity which contains $\gamma$. Then for this special case 
(i.e., $a(\chi_1)=a(\chi_2)=a(\chi_1\chi_2)=1$), we can write:
\begin{equation}
 \epsilon(\chi_1\chi_2)\equiv \epsilon(\chi_1,\psi)\cdot\epsilon(\chi_2,\psi)\mod{\mu}.
\end{equation}

\end{rem}

We also observe that our local Jacobi sum $J_1(\chi_1,\chi_2,n)$ is the generalization of the classical 
Jacobi sum. But explicit computation of this local Jacobi sums are difficult. When $n=1=a(\chi_1)=a(\chi_2)$, from equation 
(\ref{eqn 2.4}) we can say that $|J_1(\chi_1,\chi_2,1)|=q^{\frac{1}{2}}$. If we can compute this local 
Jacobi sums explicitly, then by using Theorem \ref{Theorem 5.1}, one can give more 
explicit twisting formula of epsilon factors. 

By using our twisting formula 
(\ref{eqn 4.8}) and Deligne's formula (\ref{eqn 3.26}), for the following case we can give the explicit
formula for local Jacobi sum $J_1(\chi_1,\chi_2,a(\chi_1))$, when $a(\chi_1)>a(\chi_2)=1$.

\begin{prop}
Let $F$ be a non-Archimedean local field with $q$ as the cardinality of the residue field of $F$.
 Let $\chi_1$ be a character of $F^\times$ of conductor $a(\chi_1)>1$. Let $\chi_2$ be a character of $F^\times$ of 
 conductor $a(\chi_2)=1$. Let $\psi$ is an additive character of $F$ of conductor $-1$. Then 
 \begin{equation}
  J_1(\chi_1,\chi_2,a(\chi_1))=q^{a(\chi_1)-1}\cdot\chi_2(y)\cdot G(\chi_{2}^{-1},\psi),
 \end{equation}
where $y=y(\chi_1,\psi)\in F^\times$ such that $\chi_1(1+x)=\psi(yx)$ for all $x\in F$ with valuation $\nu_F(x)\ge \frac{a(\chi_1)}{2}$. 
\end{prop}

\begin{proof}
 For the assumptions, from equation \ref{eqn 4.8} we have 
 \begin{equation}\label{eqn 4.20}
  \epsilon(\chi_1\chi_2,\psi)=\frac{q^{a(\chi_1)-\frac{1}{2}}\epsilon(\chi_1,\psi)\epsilon(\chi_2,\psi)}{J_1(\chi_1,\chi_2,a(\chi_1))}.
 \end{equation}
Since $a(\chi_1)>1=a(\chi_2)$, hence $a(\chi_1)\ge 2\cdot a(\chi_2)$, from the Deligne's formula (\ref{eqn 3.26}) we have 
\begin{equation}\label{eqn 4.21}
 \epsilon(\chi_1\chi_2,\psi)=\chi_{2}^{-1}(y)\cdot\epsilon(\chi_1,\psi),
\end{equation}
where $y=y(\chi_1,\psi)\in F^\times$ such that $\chi_1(1+x)=\psi(yx)$ for all $x\in F$ with valuation $\nu_F(x)\ge \frac{a(\chi_1)}{2}$.

Comparing equations (\ref{eqn 4.20}) and (\ref{eqn 4.21}) we have 
\begin{equation}\label{eqn 4.22}
 J_1(\chi_1,\chi_2,a(\chi_1))=q^{a(\chi_1)-\frac{1}{2}}\cdot\chi_2(y)\cdot\epsilon(\chi_2,\psi).
\end{equation}
Again by the given conditions $a(\chi_2)=1$, and $n(\psi)=-1$, hence $c=\pi_{F}^{a(\chi_2)+n(\psi)}=\pi_{F}^{1-1}=1$. Thus we can write 
\begin{equation}\label{eqn 4.23}
 \epsilon(\chi_2,\psi)=q^{-\frac{1}{2}}\sum_{x\in U_F/U_{F}^{1}}\chi_{2}^{-1}(x)\psi(x)=q^{-\frac{1}{2}}\cdot G(\chi_{2}^{-1},\psi),
\end{equation}
since $\psi|_{O_F}$ is an additive character of $O_F/P_F$.

By using equation (\ref{eqn 4.23}), from equation (\ref{eqn 4.22}) we obtain:
 \begin{equation*}
  J_1(\chi_1,\chi_2,a(\chi_1))=q^{a(\chi_1)-1}\cdot\chi_2(y)\cdot G(\chi_{2}^{-1},\psi).
 \end{equation*}
\end{proof}

\vspace{20mm}
\textbf{Acknowledgements} 
\thispagestyle{empty}
I would like to thank Prof E.-W. Zink for encouraging me to work on local epsilon factors and his constant 
valuable advices. I also express my gratitude to my adviser Prof. Rajat Tandon for his continuous help.
I also thank to the referee for his/her valuable comments and suggestions for the improvement of the article.

\end{document}